\documentclass{amsart}
\usepackage{amssymb}
\usepackage{amsfonts}

\newtheorem{thm}{Theorem}[section]
\theoremstyle{definition}
\newtheorem{cor}[thm]{Corollary}
\newtheorem{prop}[thm]{Proposition}
\newtheorem{defn}[thm]{Definition}

\newtheorem{rem}[thm]{Remark}
\newtheorem{ex}[thm]{Example}

\numberwithin{equation}{section}
\begin{document}

\title {Cubes-difference factor absorbing ideals of a commutative ring}

\author{Faranak Farshadifar}

\date{\today}


\begin{abstract}
Let $R$ be a commutative ring with $1 \neq 0$.
The purpose of this paper is to introduce and investigate cubes-difference factor absorbing ideals of $R$ as a generalization of prime ideals.
We say that a proper ideal $I$ of $R$ is a {\it cubes-difference factor absorbing ideal} (\textit{cdf-absorbing ideal}) of $R$ if whenever  $a^3 - b^3 \in I$ for $ a, b \in R$, then $a - b \in I$ or  $a^2+ab+b^2 \in I$.
\end{abstract}

 \maketitle
%%%%%%%%%
%%%%%%%%%%%%
%%%%%%%%%%%%%%%%%%%
%%%%%%%%%%%%%%%%%%%%%%%
%%%%%%%%%%%%%%%%%%%%%%%%%%%
%%%%%%%%%%%%%%%%%%%%%%%%%%%%%%%
%%%%%%%%%%%%%%%%%%%%%%%%%%%%%%%%%%%
\section{Introduction} \label{s1}
\noindent
Throughout this paper, $R$ will denote a commutative ring with
identity $1 \neq 0$ and $\Bbb Z$ will denote the ring of integers. Also, $\Bbb Z_n$ will denote the  ring of integers modulo $n$ for positive integer $n$.

A proper ideal $P$ of $R$ is said to be a \textit{prime ideal} if $ab \in P$ for some
$a, b \in R$, then either $a \in P$ or $b \in P $ \cite{MR3525784}.
Theory of prime ideals is an important tool in classical algebraic geometry.
In development of algebraic geometry, some generalizations
for the concept of prime ideals has arisen for some examples see \cite{MR4127389, ABC2025, AB2, AB,  B,  Faranak2024}.
In \cite{ABC2025},  the authors introduced and studied the
notion of square-difference factor absorbing ideals of a commutative ring. A proper ideal $I$ of $R$ is said to be a \textit{square-difference factor absorbing ideal} (\textit{sdf-absorbing ideal}) of $R$ if for $0\not= a, b \in R$, whenever
$a^2 - b^2 \in  I$, then $a + b \in  I$ or $a -b \in  I$ \cite{ABC2025}.

Motivated by sdf-absorbing ideals, the aim of this paper is to introduce the notion of
{\it cubes-difference factor absorbing ideals} (\textit{cdf-absorbing ideals}) of $R$ as a generalization of prime ideals and obtain some results similar to the results on sdf-absorbing ideals of a commutative ring  given in \cite{ABC2025}.
We say that a proper ideal $I$ of $R$ is a {\it cubes-difference factor absorbing ideal} (\textit{cdf-absorbing ideal}) of $R$ if whenever  $a^3 - b^3 \in I$ for $ a, b \in R$, then $a - b \in I$ or  $a^2+ab+b^2 \in I$.

%%%%%%%%%
%%%%%%%%%%%%
%%%%%%%%%%%%%%%%%%%
%%%%%%%%%%%%%%%%%%%%%%%
%%%%%%%%%%%%%%%%%%%%%%%%%%%
%%%%%%%%%%%%%%%%%%%%%%%%%%%%%%%
%%%%%%%%%%%%%%%%%%%%%%%%%%%%%%%%%%%
\section{Main Results}
We begin with the definition.
\begin{defn}
We say that a proper ideal $I$ of $R$ is a {\it cubes-difference factor absorbing ideal} (\textit{cdf-absorbing ideal}) of $R$ if whenever  $a^3 - b^3 \in I$ for $ a, b \in R$, then $a - b \in I$ or  $a^2+ab+b^2 \in I$.
\end{defn}

\begin{rem}
This is clear by using the fact that $a^3+b^3=a^3-(-b)^3$, an ideal $I$ of $R$ is a cubes-difference factor absorbing ideal of $R$ if and only if whenever  $a^3 + b^3 \in I$ for $ a, b \in R$, then $a + b \in I$ or  $a^2-ab+b^2 \in I$.
\end{rem}

\begin{rem}
Clearly, every prime ideal of $R$ is a cdf-absorbing ideal of $R$. But the Example \ref{e1} (a) shows that the converse is not true in general.
\end{rem}

\begin{prop}\label{t1}
Let $I$ be a cdf-absorbing ideal of $R$. Then for each $a \in R$ we have that $a^3 \in I$ implies either $a^2 \in I$ or  $a \in I$.
\end{prop}
\begin{proof}
Let  $a^3 \in I$. If $a=0$, we are done. So, let $a\not=0$.
Then for each $ i \in I$, we have $a^3 - i^3 \in I$. This implies that $a - i \in I$ or $a^2+ai+i^2 \in I$ since $I$ is a cdf-absorbing ideal of $R$. Hence $a \in I$ or $a^2 \in I$.
\end{proof}

The following example shows that an ideal $I$ with the property that for each $a \in R$, $a^3 \in I$ implies either $a^2 \in I$ or  $a \in I$ need not be a cdf-absorbing ideal. However, the converse of Proposition~\ref{t1} does hold when char$(R) = 3$.
\begin{ex}
 Consider the ideal $35\Bbb Z$ of the ring $\Bbb Z$. Then for $a=3$ and $b=-2$ of $\Bbb Z$ we have
 $a^3-b^3=35 \in 35\Bbb Z$. But $a-b=5 \not \in 35\Bbb Z$ and $a^2+ab+b^2=7 \not \in 35\Bbb Z$. Thus
$35\Bbb Z$ is not a cdf-absorbing ideal of $\Bbb Z$. But since $35\Bbb Z$ is a radical ideal of the ring $\Bbb Z$,
for each $a \in \Bbb Z$, $a^3 \in 35\Bbb Z$ implies either $a^2 \in 35\Bbb Z$ or  $a \in 35\Bbb Z$.
\end{ex}

\begin{thm}\label{t2}
Let $I$ be an ideal of $R$ with the property that for each $a \in R$, $a^3 \in I$ implies either $a^2 \in I$ or  $a \in I$ and char$(R) = 3$. Then $I$ is a cdf-absorbing ideal of $R$.
\end{thm}
\begin{proof}
Let $a^3 - b^3 \in I$ for $ a, b \in R$. Since char$(R) = 3$, we have $(a - b)^3 = a^3-3a^2b+3ab^2-b^3=a^3-b^3\in I$, and thus either $a -b \in I$ or $(a-b)^2\in I$ by assumption. So, we have $(a-b)^2\in I$.
Since char$(R) = 3$, we have $a^2-2ab=a^2+ab$. Thus $a^2+ab+b^2=a^2-2ab+b^2=(a-b)^2\in I$.
Therefore, $I$ is a cdf-absorbing ideal of $R$.
\end{proof}

\begin{cor}\label{c99}
Let $R$ be a commutative von Neumann regular ring with char($R) = 3$. Then every proper ideal of $R$ is a cdf-absorbing ideal of $R$.
\end{cor}
\begin{proof}
Every proper ideal $I$ of a commutative von Neumann regular ring is a radical ideal and so
for each $a \in R$ we have that $a^3 \in I$ implies either $a^2 \in I$ or  $a \in I$.
Now, the result follows from the Theorem \ref{t2}.
\end{proof}

\begin{ex} \label{e10}
\begin{itemize}
\item [(a)]  All proper non-zero ideals of $\mathbb{Z}_8$ are cdf-absorbing ideals.
\item [(b)] Consider the ideal $8\Bbb Z$ of the ring $\Bbb Z$. Then for $a=4$ and $b=2$ of $\Bbb Z$ we have $a^3-b^3=56\in 8\Bbb Z$. But $a-b=2\not \in 8\Bbb Z$ and $a^2+ab+b^2 =28\not \in 8\Bbb Z$. Thus $8\Bbb Z$ is not a cdf-absorbing ideal of $\Bbb Z$. Also, for this $a$ and $b$ we have the zero ideal of $\mathbb{Z}_8$ is not cdf-absorbing ideal.
\item [(c)] For $a=(4,0)$ and $b=(1,0)$ of $\mathbb{Z}_9 \times \mathbb{Z}_9$ we have $a^3-b^3=(63,0)=(0,0) \in \{0\} \times \mathbb{Z}_9$. But
$a-b=(3.0)\not \in \{0\} \times \mathbb{Z}_9$ and $a^2+ab+b^2 =(21,0)\not \in \{0\} \times \mathbb{Z}_9$. Thus the ideal $\{0\} \times \mathbb{Z}_9$ is not a cdf-absorbing ideal of $\mathbb{Z}_9 \times \mathbb{Z}_9$. Also, for this $a$ and $b$, we get that the zero ideal of $\mathbb{Z}_9 \times \mathbb{Z}_9$ is not a cdf-absorbing ideal.
\item [(d)] Let $R = \mathbb{Z}_9 \times \mathbb{Z}_9 \times \mathbb{Z}_9$. Then the ideal $I = \{0\} \times \{0\} \times \mathbb{Z}_9$ is not a cdf-absorbing ideal of $R$ (to see this, let $a = (2,1,0), b = (8,1,0) \in R$).
\item [(e)] Consider the ideal $9\Bbb Z$ of the ring $\Bbb Z$. Then for $a=1$ and $b=-2$ of $\Bbb Z$ we have $a^3-b^3=9\in 9\Bbb Z$. But $a-b=3\not \in 9\Bbb Z$ and $a^2+ab+b^2 =3\not \in 9\Bbb Z$. Thus $9\Bbb Z$ is not a cdf-absorbing ideal of $\Bbb Z$.
\end{itemize}
\end{ex}

\begin{thm}\label{l1}
Let $I$ be a cdf-absorbing ideal of a commutative ring $R$. Then the following statements are equivalent:
\begin{itemize}
\item [(a)] If $a^3 - b^3 \in I$ for  $ a, b \in R$, then $(a-b)^2\in I$ and $a^2+ab+b^2 \in I$;
\item [(b)] $3 \in I$;
\item [(c)] char$(R/I) = 3$.
\end{itemize}
\end{thm}
\begin{proof}
	$(a) \Rightarrow (b)$ Let $a = b = 1$. Then $a^3 - b^3 = 0 \in I$, and thus $3 = 1 + 1(1) +1= a^2+ab+b^2 \in I$ by hypothesis.
	
	$(b) \Rightarrow (a)$
 Assume that $a^3 - b^3 \in I$ for $ a, b \in R$.   Then $a - b \in I$ or $a^2+ab+b^2 \in I$ since $I$ is a cdf-absorbing ideal of $R$. If $a - b \in I$, then $a^2+ab+b^2=(a-b)^2+3ab\in I$.
If $a^2+ab+b^2 \in I$, then $(a-b)^2=a^2-2ab+b^2=a^2+ab+b^2-3ab\in I$.
 Therefore,  $(a-b)^2, a^2+ab+b^2 \in I$.

$(b) \Leftrightarrow (c)$ This is clear.
\end{proof}

\begin{defn}\label{dd3}
We say that a proper ideal $I$ of $R$ is a \textit{$*$-prime ideal} if whenever $b(a^2+ab+b^2)\in I$  for $a,b \in R$, then $b\in I$ or $a^2+ab+b^2\in I$.
\end{defn}

\begin{ex}
 Consider the ideal $39\Bbb Z$ of the ring $\Bbb Z$. Then for $a=1$ and $b=3$ of $\Bbb Z$ we have
 $b(a^2+ab+b^2)=3(13)=39 \in 39\Bbb Z$. But $b=3 \not \in 39\Bbb Z$ and $a^2+ab+b^2=13 \not \in 39\Bbb Z$. Thus
$39\Bbb Z$  is not a $*$-prime ideal of the ring $\Bbb Z$.  By Example \ref{e1} (a) $39\Bbb Z$ is a cdf-absorbing ideal of $\Bbb Z$.
\end{ex}

\begin{thm}\label{t3}
Let $I$ be an ideal of $R$ and $3 \in U(R)$. Then $I$  is a cdf-absorbing ideal of $R$ if and only if $I$ is a
$*$-prime ideal of $R$.
\end{thm}
\begin{proof}
First suppose that $I$ is a cdf-absorbing ideal of $R$ and $b(a^2+ab+b^2)\in I$ for $a,b \in R$.  First, assume that $b \not = a$ and $b \not = -a$. Let $x:= (a + 2b)/3\in R$ and $y := (a -b)/3 \in R$. Since $b \not = a$ and $b \not = -a$, we have $x^3 - y^3 = b(a^2+ab+b^2)/3 \in I$. Thus $b  \in I$ or  $a^2+ab+b^2  \in I$ since $I$ is a cdf-absorbing ideal of $R$. Next, assume that $b = a$ or $b = -a$. Then $b^3 \in I$, and hence $b \in I$ or $b^2\in I$ by Proposition \ref{t1}.

Conversely, let $I$ be a $*$-prime ideal of $R$ and $a^3-b^3 \in I$ for some $a, b \in R$.
Set  $x:= (a + 2b)/3\in R$ and $y := (a -b)/3 \in R$. Then we get that $b=x-y$ and $a=2y+x$. Thus
$$
9y(y^2+xy+x^2)=(a-b)(a^2+ab+b^2)=(a^3 -b^3)\in I.
$$
Now, since $3 \in U(R)$, we have that $y(y^2+xy+x^2)\in I$
Hence by assumption, $ (a -b)/3=y \in I$ or $(a^2+ab+b^2)/3=(y^2+xy+x^2) \in I$. Therefore, we get that $ (a -b)\in I$ or $(a^2+ab+b^2)\in I$, as needed.
\end{proof}

\begin{thm} \label{t4}
Let $I$ be a proper ideal of $R$. Then the following statements are equivalent:
\begin{itemize}
\item [(a)] $I$ is a cdf-absorbing ideal of $R$;
\item [(b)] If $b(a^2+ab+b^2) \in I$ for $a, b \in R \setminus I$, then the system of linear equations $b=Y-X$, $a=Y+2X$ has no non-zero solution in $R$ (i.e., there are no $ x, y \in R$ that satisfy both equations)
\end{itemize}	
\end{thm}
\begin{proof}
$(a) \Rightarrow (b)$
Suppose that $I$ is a cdf-absorbing ideal of $R$, $b(a^2+ab+b^2) \in I$ for $a, b \in R \setminus I$, and the system of linear equations $b=Y-X$, $a=Y+2X$ has a solution in $R$ for some $ x, y \in R$. Then $x^3 - y^3 = b(a^2+ab+b^2) \in I$, but $a=2x + y = a^2+ab+b^2 \not \in I$ and $y - x = b \not \in I$, a contradiction.
	
$(b) \Rightarrow (a)$
Suppose that $x^3 - y^3 \in I$ for $ x, y \in R$. Let $a = 2x + y$ and $b = y -x$. Then $b(a^2+ab+b^2) = x^3 - y^3  \in I$, and the system of linear equations $2X + Y = a, Y - X = b$ has a solution in $R$ for $ x, y \in R$. Thus $2x + y = a^2+ab+b^2 \in I$ or $y - x = b \in I$, and hence $I$ is a cdf-absorbing ideal of $R$.
\end{proof}

We next give several examples of cdf-absorbing ideals.

\begin{ex} \label{e1}
\begin{itemize}
\item [(a)] One can easily verify that a proper ideal $I$ of $\mathbb{Z}$ with the form $I = 3p\mathbb{Z}$, where $p\not=3$ is prime, is a cdf-absorbing ideal of $\mathbb{Z}$.
\item [(b)]
 Let $R$ be a boolean ring. Then every proper ideal of $R$ is a cdf-absorbing ideal of $R$ since $x^3 = x$ for every $x \in R$.
\item [(c)]
 The ideal $12\Bbb Z$ of the ring $\mathbb{Z}$ is a cdf-absorbing ideal of $\mathbb{Z}$.
\item [(d)]
 Let $R = K[X]$, where $K$ is a field, and $I = (X^2+X + 1)(X - 1)R$. Then
 since $X^3-1^3=(X - 1)(X^2+X + 1)\in I$ but $(X - 1)\not\in I$ and $(X^2+X + 1)\not\in I$ we have that $I$ is not a cdf-absorbing ideal of $R$..
\end{itemize}
\end{ex}

\begin{thm} \label{t349}
Let $I$ be a cdf-absorbing ideal of $R$ and let $S$ be a multiplicatively closed subset of $R$ with $I \cap S = \emptyset$. Then $S^{-1}I$ is a cdf-absorbing ideal of $S^{-1}R$.
\end{thm}
\begin{proof}
Let $(a/s)^3-(b/t)^3 \in S^{-1}I$ for some $a/s, b/t \in S^{-1}R$. Then there exists $u \in S$ such that $ua^3t^3-ub^3s^3 \in I$.
This implies that $(uat)^3-(ubs)^3 \in I$. Thus by assumption $uat-ubs \in I$ or $(uat)^2+u^2atbs+(ubs)^2 \in I$. Therefore,
$a/s-b/t \in S^{-1}I$ or $(a/s)^2+(ab)/(st)+(b/t)^2 \in S^{-1}I$, as needed.
\end{proof}

\begin{thm}\label{t7}
Let $f : R \longrightarrow T$ be a homomorphism of commutative rings. Then we have the following.
\begin{itemize}
\item [(a)] If $J$ is a cdf-absorbing ideal of $T$, then $f^{-1}(J)$ is a cdf-absorbing ideal of $R$.
\item [(b)] If $f$  is surjective and $I$ is a cdf-absorbing ideal of $R$ containing $ker(f)$, then $f(I)$ is a cdf-absorbing ideal of $T$.
\end{itemize}
\end{thm}
\begin{proof}
(a) Let $J$ be a cdf-absorbing ideal of $T$.
Assume that $x, y \in R$ such that $x^3 - y^3\in  f^{-1}(J)$. Then
$f(x)^3 - f(y)^3\in  J$. Thus by assumption, $f(x) - f(y)\in  J$ or $f(x)^2 +f(xy)+ f(y)^2\in  J$. Hence,
 $f(x - y)\in  J$ or $f(x^2 +xy+ y^2)\in  J$. Therefore, $x - y\in  f^{-1}(J)$ or $x^2 +xy+ y^2\in  f^{-1}(J)$, as needed.

(b) Let $f$ be surjective and $I$ be a cdf-absorbing ideal of $R$ containing $ker(f)$.
Assume that $x, y \in T$ such that $x^3 - y^3\in f(I)$. As $f$ is surjective, we have $x=f(a), y=f(b)$ for some $a, b \in R$. Hence,
$f(a^3-b^3)\in f(I)$. This implies that $a^3-b^3 \in I$ because $ker(f)\subseteq I$. Thus by assumption,
$a-b \in I$ or $a^2+ab+b^2 \in I$. Therefore, $x - y\in f(I)$ or $x^2+xy+ y^2\in f(I)$, as required.
\end{proof}

\begin{cor} \label{c2}
\begin{itemize}
\item [(a)]  Let $R \subseteq T$ be an extension of commutative rings and $J$ a cdf-absorbing ideal of $T$. Then  $J \cap R$ is a cdf-absorbing ideal of $R$.
\item [(b)] If $J  \subseteq I$, then $I/J$ is a cdf-absorbing ideal of $R/J$ if and only if $I$ is a cdf-absorbing ideal of $R$.
\end{itemize}
\end{cor}
\begin{proof}
This follows from Theorem \ref{t7}.
\end{proof}

The following example shows that the ``$ker(f) \subseteq I$'' hypothesis is needed in Theorem~\ref{t7} (b).
	
\begin{ex}\label{e2}
Let $f : \mathbb{Z}[X] \longrightarrow \mathbb{Z}$ be the epimorphism given by $f(g(X)) = g(0)$. Then $I = (X + 8)$ is a prime ideal, and thus a cdf-absorbing ideal, of $\mathbb{Z}[X]$, but  $f((X+8)) = 8\mathbb{Z}$ is not a cdf-absorbing ideal of $\mathbb{Z}$ by Example \ref{e10} (b). Note that $ker(f) = (X) \not \subseteq (X + 8) = I$; so the ``$ker(f) \subseteq I$'' hypothesis is needed in Theorem~\ref{t7} (b).
\end{ex}

\begin{thm}\label{t999}
 Let $I_1, I_2$ be non-zero proper ideals of the commutative rings $R_1, R_2$, respectively.
 \begin{itemize}
\item [(a)] If $I_1 \times I_2$ is a cdf-absorbing ideal of $R_1 \times R_2$, then  $I_1, I_2$ are cdf-absorbing ideals of $R_1, R_2$, respectively. Also, if $3 \not \in I_2$ and $\sqrt[3]{i+1}\in R_1$ for some $0\not=i\in I_1$, then $\sqrt[3]{i+1}-1 \in I_1$.
\item [(b)] If $I_1, I_2$ are cdf-absorbing ideals of $R_1, R_2$, respectively, and $3 \in I_2$, then for each $(0, 0) \not = a = (a_1, a_2), b = (b_1, b_2) \in R$ with $a^3 - b^3 \in I$, we have $a^2 +ab+ b^2 \in I$ or  $(a - b)^2 \in I$.
\end{itemize}
\end{thm}
\begin{proof}
$(a) \Rightarrow (b)$
Let  $I = I_1 \times I_2$ be a cdf-absorbing ideal of $R = R_1 \times R_2$. Then one can see that $I_1, I_2$ are cdf-absorbing ideals of $R_1$, $R_2$, respectively. Next, assume that $\sqrt[3]{i+1}\in R_1$ for some $0\not=i\in I_1$ and $3\not \in I_2$. Let $a = (\sqrt[3]{i+1}, 1), b = (1, 1) \in R$. Then $a^3 - b^3 =  (i, 0) \in I$; so $(\sqrt[3]{i+1}-1,0) = a - b  \in I$ or $(\sqrt[3]{(i+1)^2}+\sqrt[3]{i+1}+1,3) = a^2+ab+b^2 \in I$. Since $3 \not\in I_2$. we get that $\sqrt[3]{i+1}-1 \in I_1$.
		
$(b) \Rightarrow (a)$
Assume that $3 \in I_2$. Let $(0, 0) \not = a = (a_1, a_2), b = (b_1, b_2) \in R$ with $a^3 - b^3 \in I$. Then $a_2^3 - b_2^3 \in I_2$, and thus $a_2 - b_2 \in I_2$ or $a_2^2+a_2b_2 + b_2^2 \in I_2$ since $I_2$ is a non-zero cdf-absorbing ideal of $R_2$.
Since $3 \in I_2$, we have $a_2^2+a_2b_2+b_2^2, (a_2 - b_2)^2 \in I_2$ by  Theorem~\ref{l1}. Also, $a_1^3 - b_1^3 \in I_1$; so $a_1 + b_1 \in I_1$ or $a_1^2+a_1b_1+ b_1^2 \in I_1$ since $I_1$ is a non-zero cdf-absorbing ideal of $R_1$.
If $a_1^2 +a_1b_1+ b_1^2 \in I_1$, then $a^2 +ab+ b^2 \in I$.
If $a_1 - b_1 \in I_1$, then $(a_1 - b_1)^2 \in I_1$. Therefore, $(a - b)^2 \in I$.
\end{proof}

In the following result, we determine the cdf-absorbing ideals in idealization rings. Recall that for a commutative ring $R$ and $R$-module $M$, the {\it idealization of $R$ and $M$} is the commutative ring $R(+)M = R \times M$ with identity $(1,0)$ under addition defined by $(r,m) + (s,n) = (r+s,m + n)$ and multiplication defined by $(r,m)(s,n) = (rs,rn + sm)$. For more on idealizations, see \cite{AW, H}.  Every ideal of $R(+)M$ has the form $I(+)N$ for $I$ an ideal of $R$ and $N$ a submodule of $M$ (see \cite[Theorem 25.1(1)]{H}).

\begin{thm}\label{t907}
Let $I$ a non-zero proper ideal of $R$, $M$ an $R$-module, and $N$ a submodule of $M$. Then we have the following.
\begin{itemize}
\item [(a)]  If $I (+) N$ is a cdf-absorbing ideal of $R (+) M$, then $I$ is a cdf-absorbing ideal of $R$ and $\{im : i\in I \ and \ m \in M \setminus N\}\subseteq N$.
\item [(b)]  If $I$ is a cdf-absorbing ideal of $R$, then $I (+) M$ is a cdf-absorbing ideal of $R (+) M$.
\end{itemize}
\end{thm}
\begin{proof}
(a) Let $I (+) N$ be a cdf-absorbing ideal of $R (+) M$. Then it is easily verified that $I$ is a cdf-absorbing ideal of $R$.  Let $m \in M \setminus N$, $ i \in I$, and $a: = (i, 0), b: = (0, m) \in R(+)M$. Then $a^3 - b^3 = (i^3,  0) \in I(+)N $ implies that $a^2+ab+b^2 = (i^2, im) \in I (+) N$  or  $a - b = (i, -m) \in I (+) N$.
Since $m \in M \setminus N$, we have  $a - b = (i, -m) \not\in I (+) N$. Thus $a^2+ab+b^2 = (i^2, im) \in I (+) N$. Hence, $im \in N$. Therefore, $\{im : i\in I \ and \ m \in M \setminus N\}\subseteq N$.

(b) Let $I$ be a cdf-absorbing ideal of $R$. Suppose that $a^3 - b^3 \in I (+) M$  for $(0, 0) \neq a, b \in R (+) M$, where $x = (a, m_1)$ and $y =(b, m_2)$. Since $I$  is a non-zero cdf-absorbing ideal of $R$ and $a^3 - b^3 \in I$, we have  $a^2+ab + b^2 \in I$ or $a - b \in I$. If $a^2+ab + b^2 \in I$, then $x^2+xy+ y^2 \in  I (+) M $. If $a - b \in I$, then $x - y \in  I (+) M $. Thus $I (+) M$ is a cdf-absorbing ideal of $R (+) M$.
\end{proof}

The following example shows that it is crucial that $I$ be a non-zero ideal in Theorem~\ref{t907}.

\begin{ex} \label{e101}
{\rm  Let $R = \mathbb{Z}_8$, $M =N = \mathbb{Z}_8$, and $I = \{0\}$. Then $\{0\}$ is a cdf-absorbing ideal of $\mathbb{Z}_8$, but $\{0\} (+) \mathbb{Z}_8$ is not a cdf-absorbing ideal of $\mathbb{Z}_8 (+) \mathbb{Z}_8$. To see this consider $x = (2,0), y = (0,2)$. Then $a^3-b^3=(8,0)=(0,0)\in\{0\} (+) \mathbb{Z}_8$ but $a-b=(2,-2) \not \in \{0\} (+) \mathbb{Z}_8$ and $a^2+ab+b^2=(4,4) \not \in \{0\} (+) \mathbb{Z}_8$.}
\end{ex}

Recall that a proper ideal $I$ of $R$ is said to be \textit{$n$-semi-absorbing} for the positive integer $n$ if $x^{n+1}\in I$ implies that $x^n \in I$ for any $x\in R$ \cite{2017}.

Next, we consider when $\{0\}(+)M$ is a cdf-absorbing ideal of $R(+)M$.
\begin{prop} \label{p101}
{ \rm Let $M$ be a non-zero $R$-module.
Then $\{0\}(+)M$ is a cdf-absorbing ideal of $R(+)M$ if and only if $\{0\}$ is a 2-semi-absorbing ideal of $R$ and $\{0\}$ is a cdf-absorbing ideal of $R$.}
\end{prop}
\begin{proof}
First assume that $\{0\}(+)M$ is a cdf-absorbing ideal of $R(+)M$. Suppose that $\{0\}$ is not a 2-semi-absorbing ideal of $R$. Then there exists $a\in R$ such that $a^3=0$ but $a^2\not=0$. Let $x=(a,0)$ and $y=(0,m)$ for some $m \in M$. Then $x^3-y^3=(a^3,0)=(0,0) \in \{0\}(+)M$.
Thus by assumption, $x-y=(a,-m)\in \{0\}(+)M$ or $x^2+xy+y^2=(a^2,am) \in \{0\}(+)M$, which are contradictions. Thus $\{0\}$ is a 2-semi-absorbing ideal of $R$. Now, let $a^3-b^3 \in  \{0\}$. Then by consider $x=(a,0)$ and $y=(b,0)$ one can see that $\{0\}$ is a cdf-absorbing ideal of $R$. Conversely,  Assume that $\{0\}$ is a cdf-absorbing ideal of $R$. Let $x^3 - y^3 \in \{0\} (+) M$  for $x, y \in R (+) M$, where $x = (a, m_1)$ and $y =(b, m_2)$. Since $\{0\}$ is a cdf-absorbing ideal of $R$ and $a^3 - b^3 \in \{0\}$, we have  $a^2+ab + b^2 \in \{0\}$ or $a - b \in \{0\}$. If $a^2+ab + b^2 \in \{0\}$, then $x^2+xy+ y^2 \in  \{0\} (+) M $. If $a - b \in \{0\}$, then $x - y \in  \{0\} (+) M $. Thus $\{0\} (+) M$ is a cdf-absorbing ideal of $R (+) M$.
\end{proof}

In the next result, we study cdf-absorbing ideals in amalgamation rings. Let $A, B$ be commutative rings, $f: A \longrightarrow B$ a homomorphism, and $J$ an ideal of $B$. Recall that the {\it amalgamation of $A$ and $B$ with respect to $f$ along $J$} is the subring $A \bowtie_J B = \{(a, f(a) + j ) \mid a \in A, j \in J\}$ of $A \times B$ \cite{DF07}.

\begin{thm}\label{t709}
Let $A$ and $B$ be commutative rings, $f: A \longrightarrow B$ a homomorphism, $J$ an ideal of $B$, and $I$ a non-zero proper ideal of $A$. Then $I \bowtie_J B$ is a cdf-absorbing ideal of $A\bowtie_J B$ if and only if $I$ is a cdf-absorbing ideal of $A$.
\end{thm}
\begin{proof}
If $I \bowtie_J B$ is a cdf-absorbing ideal of $A\bowtie_J B$, then it is easily to see that $I$ is a cdf-absorbing ideal of $A$.

Conversely, assume that $I$ is a non-zero cdf-absorbing ideal of $A$. Let $x = (a, f(a)  + j_1), y = (b, f(b) + j_2) \in A\bowtie_J B$ such that $x^3 - y^3 \in I \bowtie_J B$. Since $a^3 - b^3 \in I$ and $I$ is a non-zero cdf-absorbing ideal of $A$, we have $a^2+ab + b^2 \in I$ or $a - b \in I$. If $a^2+ab+ b^2\in I$, then
$$
x^2+xy+ y^2 =
$$
$$
(a^2+ab + b^2, f(a^2+ab + b^2) + j_1(2f(a)+j_1+f(b)) +j_2(2f(b)+j_2+f(a)))\in I\bowtie_J B.
$$
Similarly, if $a - b \in I$, then
$$
x - y = (a - b, f(a -b) + j_1 - j_2) \in I\bowtie_J B.
$$
Thus $I\bowtie_J B$ is a cdf-absorbing ideal of $A\bowtie_J B$.
\end{proof}

The following example shows that it is again crucial that $I$ be a non-zero ideal in Theorem~\ref{t709}.

\begin{ex}
{\rm Let $A = B = J = \mathbb{Z}_8$, $f = 1_A : A \longrightarrow A$, and $I = \{0\}$. Then $\{0\}$ is a cdf-absorbing ideal of $\mathbb{Z}_8$. But $\{0\} \bowtie_{\mathbb{Z}_8} \mathbb{Z}_8$ is not a cdf-absorbing ideal of $\mathbb{Z}_8\bowtie_{\mathbb{Z}_8} \mathbb{Z}_8$. To see this, consider $x = (2,0), y = (0,2)$. Then $a^3-b^3=(8,8)=(0,0)\in \{0\} \bowtie_{\mathbb{Z}_8} \mathbb{Z}_8$ but $a-b=(2,-2) \not \in \{0\} \bowtie_{\mathbb{Z}_8} \mathbb{Z}_8$ and $a^2+ab+b^2=(4,4) \not \in \{0\} \bowtie_{\mathbb{Z}_8} \mathbb{Z}_8$.}
\end{ex}

\end{document}